\documentclass[times, doublespace,a4]{amsart}
\usepackage{natbib}
  \usepackage{amsmath}
\usepackage{amsfonts}
\usepackage{amssymb}
\usepackage{amsthm}
\newtheorem{lemma}{Lemma}
 \newtheorem{theorem}{Theorem}
 \newtheorem{proposition}{Proposition}
 \newtheorem{corollary}{Corollary}
 \newtheorem{remark}{Remark}

 \theoremstyle{example}
\newtheorem{example}{Example}

\begin{document}
 
 \title{Complex multiplicative calculus}
 \author[A. Bashirov]{Agamirza E. Bashirov}
 \author[M. Riza]{Mustafa Riza}
 \email{mustafa.riza@emu.edu.tr}
 
 \address{
 Department of Mathematics, Eastern Mediterranean University, Gazima\u{g}usa, North Cyprus. , phone: +903926301006, FAX: +903923651604}
 
 
 \begin{abstract}

 In the present paper we extend the concepts of multiplicative derivative and integral to
complex-valued functions of complex variable. Some drawbacks, arising with these concepts in the real case, are
explained satisfactorily. Properties of complex multiplicative derivatives and integrals are studied. In particular,
the fundamental theorem of complex multiplicative calculus, relating these concepts, is proved. It is shown that
complex multiplicative calculus is not just another realization of the ordinary calculus. In particular, the Cauchy
formula of complex calculus disappears in multiplicative complex calculus.
\end{abstract}

\maketitle

\section{Introduction}

The distinction of two real numbers $a$ and $b$ with $a\le b$ can be measured in different ways. The most popular way
is the difference $h=b-a$, saying that $b$ is greater than $a$ for $h$ units. Based on this in the second half of the
17th century Isaac Newton and Gottfried Wilhelm Leibnitz created differential and integral calculus. Later this
calculus was adapted to the study of functions by Leonard Euler.

Another way to measure the distinction of $a$ and $b$, if $0< a\le b$, is the ratio $r=b/a$, saying that $b$ is $r$
times as greater as $a$. This way gave rise to an alternative calculus, called multiplicative calculus, in the work of
Michael Grossman and Robert Katz \cite{GK}. Further contribution to multiplicative calculus and its applications was
done in Stanley \cite{ST}, Bashirov et al. \cite{B} and Riza et al. \cite{ROK}, etc. Independently, some elements of
stochastic multiplicative calculus are developed in the works of Karandikar \cite{K} and Daletskii and Teterina
\cite{DT}. There are other works by Volterra and Hostinsky \cite{V}, Aniszewska \cite{A}, Kasprzak et al. \cite{KLR},
Rybaczuk et al. \cite{RKZ}, which employ the term of multiplicative calculus but indeed they concern bigeometric
calculus in the Grossman's terminology \cite{G}. The bigeometric calculus is also considered in C\'{o}rdova-Lepe
\cite{C} under the name of proportional calculus. The Volterra's product integral from \cite{V} found further
development in the book by Slavik \cite{SL}.

In the existing literature, multiplicative calculus is considered as a calculus applicable to positive functions,
creating several questions. For example, the multiplicative derivative can also be extended to negative functions while
the functions with both positive and negative values are left out of multiplicative calculus. All these questions
suggest, that there may be an extension of multiplicative calculus, explaining this issue. Logically, such an extension
is seen in a development of multiplicative calculus for complex-valued functions of complex variable, i.e., creation of
complex multiplicative calculus. There are a lot of books on complex calculus. We refer to Ahlfors \cite {AH}, Conway
\cite{CO}, Sarason \cite{S}, Greene and Krantz \cite{GKr}, Lang \cite{L}, Derrick \cite{D} and Palka \cite{P}, which are
used during this study.

In the present paper the concepts of differentiation and integration of complex multiplicative calculus are defined and
studied. They are interpreted geometrically and a relationship between them is established. The obstacles from the real
case are explained satisfactorily. The developments show that complex multiplicative calculus is not just another
realization of the ordinary complex calculus. In particular, it is found that in complex multiplicative calculus the
Cauchy formula disappears, since the inequality $2\pi i\not= 0$ takes place as the equality $e^{2\pi i}=e^0=1$. This
demonstrates the importance of further study in complex multiplicative calculus.

One major remark about the notation is that the multiplicative versions of the concepts of Newtonian calculus will be
called as *concepts, for example, a *derivative means a multiplicative derivative. We denote by $\mathbb{R}$ and
$\mathbb{C}$ the fields of real and complex numbers, respectively. $\mathbb{N}$ and $\mathbb{Z}$ denote the collection
of all natural numbers and all integers, respectively. $\mathrm{Arg}\,z$ is the principal value of $\arg z$, noticing
that $-\pi <\mathrm{Arg}\,z\le \pi$. Always $\ln x$ refers to the natural logarithm of the real number $x>0$ whereas
$\log z$ to the same of the complex number $z\not= 0$. By $\mathrm{Log}\,z$ we denote the value at $z$ of the principal
branch of the complex logarithmic function, i.e., $\mathrm{Log}\,z=\ln |z|+i\mathrm{Arg}\,z$, where $i$ is the
imaginary unit and $|z|$ is the modulus of $z$. Under a function we always mean a single-valued function. The cases of
multi-valued functions are pointed out.


 \section{Motivation}

The *derivative at $t$ of a real-valued function $f$ of real variable is defined as the limit
 \begin{equation}
 \label{1}
 f^*(t)=\lim _{h\to 0}(f(t+h)/f(t))^{1/h},
 \end{equation}
that shows how many times the value $f(t)$ changes at $t$. If $f$ has pure positive values and is differentiable at
$t$, then in Bashirov et al. \cite{B} it is shown that $f^*$ exists and is related to the ordinary derivative $f'$ as
 \[
 f^*(t)=e^{[\ln f]'(t)}=e^{f'(t)/f(t)}.
 \]
One can also observe that if $f$ has pure negative values and is differentiable at $t$, then the limit in (\ref{1})
still exists with
 \[
 f^*(t)=e^{[\ln |f|]'(t)}=e^{|f|'(t)/|f|(t)}.
 \]
In both cases $f^*(t)>0$ and all higher order *derivatives of $f$ are defined as the first order
*derivatives of positive functions. Based on this, in Bashirov et al. \cite{B} multiplicative calculus was presented as
a calculus for positive functions, raising the following questions:
 \begin{enumerate}
 \item [(a)] Why all order *derivatives of a negative function are positive functions?
 \item [(b)] Why *differentiation is not applicable to functions with both positive and negative values?
 \item [(c)] What is the role of zero in *differentiation?
 \end{enumerate}

Similar questions arise with *integral as well. These questions from the real case suggest that there may be an
extension of *derivative, bringing an explanation to this issue. Indeed, all these questions find satisfactory answers
in the complex case.

For motivation, consider a complex-valued function $f$ of real variable on an open interval $(a,b)$. Considering the
limit in (\ref{1}) as a *derivative of $f$ at $t$, we have
 \begin{align*}
 f^*(t) &=\lim _{h\to 0}(f(t+h)/f(t))^{1/h}=\lim _{h\to 0}e^{(1/h)\log (f(t+h)/f(t))}\\
        &=e^{\lim _{h\to 0}(1/h)\ln |f(t+h)/f(t)|+i\lim _{h\to 0}(1/h)(\mathrm{Arg}\,(f(t+h)/f(t))+2\pi n)}.
 \end{align*}
 \begin{remark}
 \label{R1}
{\rm In the existing literature on complex calculus the symbol $e^z$ has two inconsistent usages. Most popularly, $e^z$
denotes the value at $z$ of the complex exponential function, that is a unique solution of $f'(z)=f(z)$ with $f(0)=1$.
This is a single-valued function and, by the Euler's formula, equals to $e^z=e^x(\cos y+i\sin y)$ if $z=x+iy$. The
second usage is the operation of raising $e$ to the complex power $z$, that results multiple values $e^{z\log e}$ or
$e^{z(1+2\pi ni)}$ for $n\in \mathbb{Z}$, in which the first usage is assumed. Later on, we will deal with raising to
complex powers. In order to avoid possible ambiguities, instead of $w^z$ (raising complex $w$ to complex power $z$) we
will prefer to write $e^{z\log w}$, reserving the symbol $e^z$ for the complex exponential function. For raising real
$x$ to real power $y$ we will still use the symbol $x^y$.}
 \end{remark}
 \begin{remark}
 \label{R2}
{\rm If $f(t)\not= 0$ and $f'(t)$ exists, then
 \[
 \lim _{h\to 0}\frac{\ln |f(t+h)/f(t)|}{h}=\lim _{h\to 0}\frac{\ln |f(t+h)|-\ln |f(t)|}{h}=[\ln |f|]'(t).
 \]
}
 \end{remark}
 \begin{remark}
 \label{R3}
{\rm Assume $f(t)\not= 0$, $f'(t)$ exists and, additionally, $\mathrm{Arg}\, f(t)\not= \pi $. Then
$\mathrm{Arg}\,(f(t+h)/f(t))=\mathrm{Arg}\,f(t+h)-\mathrm{Arg}\,f(t)$ for all $|h|<\varepsilon $, where $\varepsilon $
is sufficiently small. This implies that
 \[
 \lim _{h\to 0}\frac{\mathrm{Arg}\,(f(t+h)/f(t))+2\pi n}{h}=\lim _{h\to 0}\frac{\mathrm{Arg}\,f(t+h)-\mathrm{Arg}\,f(t)+2\pi n}{h}
 \]
exists and equals to $[\mathrm{Arg}\,f]'(t)$ if and only if $n=0$. Therefore, it seems to be reasonable to understand
$e^{(1/h)\log (f(t+h)/f(t))}$ as its principal value.}
 \end{remark}
 \begin{remark}
 \label{R4}
{\rm Again, assume that $f(t)\not= 0$ and $f'(t)$ exists. If $\mathrm{Arg}\, f(t)=\pi $, then for small values of $h$,
 \[
 \mathrm{Arg}\,(f(t+h)/f(t))=\left\{ \!\!\begin{array}{ll}
   \mathrm{Arg}\,f(t+h)-\mathrm{Arg}\,f(t) & \text{if}\ \mathrm{Arg}\,f(t+h)\ge 0,\\
   \mathrm{Arg}\,f(t+h)-\mathrm{Arg}\,f(t)+2\pi & \text{if}\ \mathrm{Arg}\,f(t+h)<0.
   \end{array}\!\!\right.
 \]
Therefore, the selection of the principal value of $e^{(1/h)\log (f(t+h)/f(t))}$, mentioned in Remark \ref{R3}, does
not guarantee the equality
 \[
 \lim _{h\to 0}(1/h)\mathrm{Arg}\,(f(t+h)/f(t))=[\mathrm{Arg}\,f]'(t)
 \]
in cases when $\mathrm{Arg}\,f(t)=\pi $. Instead, the issue can be improved if we replace $\mathrm{Arg}\, f$, ranging
in the interval $(-\pi ,\pi ]$, by any other branch $\Theta $ of $\arg f$, ranging in $(-\pi +\alpha ,\pi +\alpha]$
with $\alpha \in \mathbb{R}$, so that $\Theta (t)\not= \pi +\alpha $. But for $\alpha \not= 0$ the principal value of
$e^{(1/h)\log (f(t+h)/f(t))}$ differs from
 \[
 e^{(\ln |f(t+h)|-\ln |f(t)|)/h+i(\Theta (t+h)-\Theta (t))/h},
 \]
refusing the development from Remark \ref{R3}. This discrepancy can be overcome if we use the limit in (\ref{1}) only
for motivation of complex *derivatives, and define $f^*(t)$ directly as $e^{[\ln |f|]'(t)+i\Theta '(t)}$. Clearly, the
value of $\Theta '(t)$ is independent on selection of the branch $\Theta $ of $\arg f$ among those which range in
$(-\pi +\alpha , \pi +\alpha ]$ for $\alpha \in \mathbb{R}$ with $\Theta(t)\not= \pi +\alpha $.}
 \end{remark}
 \begin{remark}
 \label{R5}
{\rm Another issue is whether $[\ln |f|]'(t)+i\Theta '(t)=[\log f]'(t)$. Generally speaking, a branch of $\log f$ may
not exist. Even if it exists, it may not be a composition of a branch of $\log $ and $f$. In this regard, no problem
arises with *derivative because it is based on the \emph{local behavior} of the function $f$. For definition of
$f^*(t)$, selecting a sufficiently small neighborhood $U\subseteq (a,b)$ of $t$, we can reach the existence of the
branches of $\log f$ locally. Additionally, we can get these branches as a composition of the respective branches of
$\log $ and $f$, again locally. Moreover, we can use the log-differentiation formula $[\log f]'=f'/f$ independently on
selection of a branch of $\log $ if $f$ is nowhere-vanishing and $f'$ exists (see, Sarason \cite{S}).}
 \end{remark}

Based on Remarks \ref{R1}--\ref{R5}, consider a differentiable nowhere-vanishing complex-valued function $f$ of real
variable on an open interval $(a,b)$. Select a small neighborhood $U\subseteq (a,b)$ of $t\in (a,b)$ such that $\log f$
on $U$ has branches in the form of composition of branches of $\log $ and $f$. Let $|f|=R$ and let $\arg f=\Theta +2\pi
n$ on $U$, where $\Theta $ is any branch of $\arg f$ on $U$, and define $f^*(t)$ as
 \[
 f^*(t)=e^{[\ln |f|]'(t)+i\Theta '(t)},
 \]
or, by the log-differentiation formula,
 \begin{equation}
 \label{2}
 f^*(t)=e^{f'(t)/f(t)}.
 \end{equation}
Defined in this way, $f^*(t)$ exists as a single complex value.

An important consequence from (\ref{2}) is that if $f$ takes values on some ray from the origin on the complex plane,
i.e.,
 \[
 f(t)=R(t)e^{i(\theta +2\pi n)},\ \theta =\mathrm{const}.,
 \]
then
 \[
 f^*(t)=e^{R'(t)e^{i(\theta +2\pi n)}/R(t)e^{i(\theta +2\pi n)}}=e^{R'(t)/R(t)}=R^*(t).
 \]
This explains why all order *derivatives of a negative function are positive. More generally, the *derivative of a
complex-valued function $f$ of real variable, taking all the values on a ray from the origin, is a positive function
and $f^*(t)$ measures how many times the distance of $f(t)$ from the origin changes at $t$.

Another important consequence from (\ref{2}) is that if $f$ takes values on some circle centered at the origin on the
complex plane, i.e.,
 \[
 f(t)=re^{i(\Theta (t)+2\pi n)},\ r=\mathrm{const}.,
 \]
then
 \[
 f^*(t)=e^{ir\Theta '(t)e^{i(\Theta (t)+2\pi n)}/re^{i(\Theta (t)+2\pi n)}}=e^{i\Theta '(t)}=e^{i(\Theta '(t)+2\pi m)},
 \]
demonstrating that $f^*$ takes values on the unit circle centered at the origin on the complex plane and one of the
multiple values of $\arg f^*(t)$, that is $\Theta '(t)$, measures the rate of change of all the branches of $\arg f$ at
$t$. In particular, $f^*(t)=-1$ if $\Theta '(t)=\pi $. This is a case when *derivative is a negative number. Also,
$f^*(t)=i$ if $\Theta '(t)=\pi /2$, a case when *derivative belongs to the imaginary axis.

More generally, if
 \[
 f(t)=R(t)e^{i(\Theta (t)+2\pi n)},
 \]
then
 \begin{align*}
 f^*(t)&=e^{(R'(t)+iR(t)\Theta '(t))e^{i(\Theta (t)+2\pi n)}/R(t)e^{i(\Theta (t)+2\pi n)}}\\
       &=e^{R'(t)/R(t)}e^{i\Theta '(t)}=R^*(t)e^{i(\Theta '(t)+2\pi m)},
 \end{align*}
i.e., the modulus and the argument of $f^*(t)$ behave similar to the above mentioned two particular cases.

The most important conclusion from (\ref{2}) is that in *differentiation the origin on the complex plane acts
differently from the other points. More precisely, instead of the complex plane $\mathbb{C}$, the perforated complex
plane $\mathbb{C}\setminus \{0\}$ should be considered. This also explains why *derivative can not be applicable to
real-valued functions, taking both positive and negative values. This is because instead of the real line $\mathbb{R}$,
the perforated real line $\mathbb{R}\setminus \{0\}$ should be considered. Consequently, this removes the real-valued
functions with both positive and negative values from the consideration, since a continuous real-valued function with
the range in $\mathbb{R}\setminus \{0\}$ is either positive or negative. The functions with both positive and negative
values become considerable if they take complex values as well. This is because bypassing the origin when traveling
continuously from positive to negative numbers and vice versa is allowed on the complex plane.

All these demonstrate that it is actual to develop *differentiation and *integration for complex-valued functions of
complex variable. This is done in the following sections.


 \section{Complex Multiplicative Differentiation}

Now assume that $f$ is a nowhere-vanishing differentiable complex-valued function on an open connected set $D$ of the
complex plane. Select a sufficiently small neighborhood $U\subseteq D$ of the point $z\in D$ such that the branches of
$\log f$ on $U$ exist in the form of the composition of the respective branches of $\log $ and the restriction of $f$
to $U$, and the log-differentiation formula is valid for $\log f$ on $U$. According to (\ref{2}), define the
*derivative of $f$ at $z\in D$ by
 \begin{equation}
 \label{3}
 f^*(z)=e^{f'(z)/f(z)}.
 \end{equation}
By induction, we also obtain the following formula for higher-order multiplicative derivatives:
\begin{equation}
 \label{4}
 f^{*(n)}(z)=e^{[f'/f]^{(n-1)}(z)}, n=1,2,\ldots .
 \end{equation}
To derive the Cauchy--Riemann conditions in *form, let $z=x+iy=re^{i\theta }$ and $f(z)=u(z)+iv(z)=R(z)e^{i(\Theta
(z)+2\pi n)}$ for $z\in U$, where $\Theta $ is any suitable branch of $\arg f$. Since $f$ is differentiable, all
functions $u$, $v$, $R$, and $\Theta $ have continuous partial derivatives in $x$, $y$, $r$, and $\theta $.
 \begin{proposition}
 \label{P1}
Under the above conditions and notation,
 \begin{equation}
 \label{5}
 R^*_x(z)=\big[ e^{\Theta }\big] ^*_y(z)\ \ \text{and}\ \ R^*_y(z)=\big[ e^{-\Theta }\big] ^*_x(z),
 \end{equation}
where $g^*_x$ and $g^*_y$ denote the partial *derivatives with respect to $x$ and $y$ of the positive function $g$.
 \end{proposition}
 \begin{proof}
It suffices to proved that
 \begin{equation}
 \label{6}
 [\ln R]'_x(z)=\Theta '_y(z)\ \ \text{and}\ \ [\ln R]'_y(z)=-\Theta '_x(z).
 \end{equation}
The proofs of these equalities are similar. Therefore, we prove one of them, for example, the second one. At first, let
$u(z)\not= 0$. From the Cauchy--Riemann conditions $u'_x=v'_y$ and $u'_y=-v'_x$,
 \begin{align*}
 [\ln R]'_y
  &=\big( \ln\sqrt{u^2+v^2}\big) '_y=\frac{1}{\sqrt{u^2+v^2}}\cdot \frac{2(uu'_y+vv'_y)}{2\sqrt{u^2+v^2}}=\frac{uu'_y+vv'_y}{u^2+v^2}\\
  &=\frac{-uv'_x+vu'_x}{u^2+v^2}=-\frac{1}{1+v^2/u^2}\cdot \frac{uv'_x-u'_xv}{u^2}\\
  &=-\Big( \tan ^{-1}\frac{v}{u}+\pi k+2\pi n\Big) '_x=-\Theta '_x,
 \end{align*}
where $k\in \{ -1,0,1\} $ and $n\in \mathbb{Z}$. Now assume $u(z)=0$. Then $v(z)\not= 0$. Consider two subcases. In the
first subcase, assume that there is a sequence $\{ z_n\} $ such that $z_n\to z$ and $u(z_n)\not= 0$. Then the
continuity of $[\ln R]'_y$ and $\Theta '_x$ implies
 \[
 [\ln R]'_y(z)=\frac{v'_y(z)}{v(z)}=\frac{u'_x(z)}{v(z)}=-\Theta '_x(z).
 \]
In the second subcase, let $u$ be identically zero on some neighborhood of $z$. Then by Cauchy--Riemann conditions, $v$
takes a nonzero constant value on this neighborhood. Hence, $[\ln R]'_y(z)=0$. At the same time, on this neighborhood
either $\Theta (z)=\pi /2+2\pi n$ or $\Theta (z)=-\pi /2+2\pi n$, implying $\Theta '_x(z)=0$. Hence, again $[\ln
R]'_y(z)=0=-\Theta '_x(z)$.
 \end{proof}
 \begin{remark}
 \label{R6}
{\rm The converse of Proposition \ref{P1} is also valid. For this, assume that the equalities in (\ref{6}), which are
equivalent to the equalities in (\ref{5}), hold. Then from $u(z)=R(z)\cos \Theta (z)$ and $v(z)=R(z)\sin \Theta (z)$,
we obtain
 \begin{align*}
 u'_x(z) &=R'_x(z)\cos \Theta (z)-R(z)\Theta '_x(z)\sin \Theta (z)\\
         &=R(z)[\ln R]'_x(z)\cos \Theta (z)-R(z)\Theta '_x(z)\sin \Theta (z)\\
         &=R(z)\Theta '_y(z)\cos \Theta (z)+R(z)[\ln R]'_y(z)\sin \Theta (z)\\
         &=R(z)\Theta '_y(z)\cos \Theta (z)+R'_y(z)\sin \Theta (z)=v'_y(z).
 \end{align*}
In a similar way $u'_y(z)=-v'_x(z)$ can be proved.}
 \end{remark}

Thus, the equalities in (\ref{5}) or in (\ref{6}) are just another form of the ordinary Cauchy--Riemann conditions and,
hence, we call them Cauchy--Riemann *conditions. They can be written in terms of partial derivatives with respect to
$r$ and $\theta $ as well. Indeed, from
 \begin{align*}
 \Theta '_\theta  &=\Theta '_xx'_\theta +\Theta '_yy'_\theta =-\Theta '_xr\sin \theta +\Theta '_yr\cos \theta \\
                  &=r([\ln R]'_yy'_r+[\ln R]'_xx'_r)=r[\ln R]'_r
 \end{align*}
and
 \begin{align*}
 [\ln R]'_\theta &=[\ln R]'_xx'_\theta +[\ln R]'_yy'_\theta =-\Theta '_yr\sin \theta -\Theta '_xr\cos \theta \\
            &=-r(\Theta '_yy'_r+\Theta '_xx'_r)=-r\Theta '_r,
 \end{align*}
we obtain
 \[
 \Theta '_\theta (z)=r[\ln R]'_r(z)\ \ \text{and}\ \ [\ln R]'_\theta (z)=-r\Theta '_r(z),
 \]
or in terms of *derivatives
 \[
 R^*_r(z)^r=\big[ e^{\Theta }\big] ^*_\theta (z)\ \ \text{and}\ \ R^*_\theta (z)=\big[ e^{-\Theta }\big] ^*_r(z)^r.
 \]
Later it will be convenient to use the Cauchy--Riemann *conditions in the following form.
 \begin{theorem}[Cauchy--Riemann *conditions]
 \label{T1}
Under the above conditions and notation,
 \begin{equation}
 \label{7}
 \left\{ \begin{array}{l}
 |f^*(z)|=R^*_x(z)=[e^{\Theta }]^*_y(z),\\
 \arg f^*(z)=\Theta '_x(z)+2\pi n=-[\ln R]'_y(z)+2\pi n,\ n\in \mathbb{Z}.
 \end{array}\right.
 \end{equation}
 \end{theorem}
 \begin{proof}
From
 \[
 f^*(z)=e^{f'(z)/f(z)}=e^\frac{u'_x(z)+iv'_x(z)}{u(z)+iv(z)}=
 e^{\frac{u(z)u'_x(z)+v(z)v'_x(z)}{u(z)^2+v(z)^2}+i\frac{u(z)v'_x(z)-u'_x(z)v(z)}{u(z)^2+v(z)^2}},
 \]
we have
 \[
 |f^*(z)|=e^{\frac{u(z)u'_x(z)+v(z)v'_x(z)}{u(z)^2+v(z)^2}}=e^{\frac{R(z)R'_x(z)}{R(z)^2}}=e^{\frac{R'_x(z)}{R(z)}}=R^*_x(z)
 \]
and, by Cauchy--Riemann conditions,
 \begin{align*}
 \arg f^*(z)&=\frac{u(z)v'_x(z)-u'_x(z)v(z)}{u(z)^2+v(z)^2}+2\pi n=-\frac{u(z)u'_y(z)+v(z)v'_y(z)}{u(z)^2+v(z)^2}+2\pi n\\
             &=-\frac{R(z)R'_y(z)}{R(z)^2}+2\pi n=-\frac{R'_y(z)}{R(z)}+2\pi n=-[\ln R]'_y(z)+2\pi n.
 \end{align*}
The other two equalities in (\ref{7}) are from (\ref{5})--(\ref{6}).
 \end{proof}

A few examples will be relevant to demonstrate features of complex *differentiation.
 \begin{example}
 \label{E1}
{\rm The function $f(z)=c$, $z\in \mathbb{C}$, where $c=\mathrm{const}\in \mathbb{C}\setminus \{ 0\} $, is an entire
function and its
*derivative
 \[
 f^*(z)=e^{f'(z)/f(z)}=e^{0/c}=1,\ z\in \mathbb{C},
 \]
is again an entire function. Respectively, $f^{*(n)}(z)=1$, $z\in \mathbb{C}$, $n\in \mathbb{N}$. Thus, in complex
*calculus the role of 0 (the neutral element of addition) is shifted to 1 (the neutral element of multiplication).}
 \end{example}
 \begin{example}
 \label{E2}
{\rm The function $f(z)=e^{cz}$, $z\in \mathbb{C}$, where $c=\mathrm{const}\in \mathbb{C}$, is an entire function and
its *derivative
 \[
 f^*(z)=e^{f'(z)/f(z)}=e^{ce^{cz}/e^{cz}}=e^c,\ z\in \mathbb{C},
 \]
is again an entire function, taking identically the nonzero value $e^c$. Respectively, $f^{*(n)}(z)=1$, $z\in
\mathbb{C}$, $n=2,3,\ldots \, $. Thus, in complex *calculus $f(z)=e^{cz}$ plays the role of the linear function
$g(z)=az$ with $a=e^c$ in Newtonian calculus.}
 \end{example}
 \begin{example}
 \label{E3}
{\rm For another entire function $f(z)=e^{ce^z}$, $z\in \mathbb{C}$, with $c=\mathrm{const}\in \mathbb{C}$, that is
called a Gompertz function if $z$ takes real values, we have
 \[
 f^*(z)=e^{f'(z)/f(z)}=e^{ce^ze^{ce^z}/e^{ce^z}}=e^{ce^z},\ z\in \mathbb{C}.
 \]
Hence, $f$ is a solution of the *differential equation $f^*=f$. Respectively, $f^{*(n)}(z)=e^{ce^z}$, $z\in
\mathbb{C}$, $n\in \mathbb{N}$. Thus, in complex *calculus $f(z)=e^{ce^z}$ plays the role of the exponential function
$g(z)=ce^z$ in Newtonian calculus.}
 \end{example}
 \begin{example}
 \label{E4}
{\rm The function $f(z)=z$, $z\in \mathbb{C}$, is also entire, but its *derivative
 \[
 f^*(z)=e^{f'(z)/f(z)}=e^{1/z},\ z\in \mathbb{C}\setminus \{ 0\},
 \]
accounts an essential singularity at $z=0$. This is because *differentiation is applicable to functions with the range
in $\mathbb{C}\setminus \{ 0\} $. Thus, the *derivative of an entire function may not be entire. We also have
$f^{*(n)}(z)=e^{(-1)^{n-1}(n-1)!/z^n}$, $z\in \mathbb{C}\setminus \{ 0\} $, $n\in \mathbb{N}$.}
 \end{example}
 \begin{example}
 \label{E5}
{\rm The holomorphic function $f(z)=1/z$, $z\in \mathbb{C}\setminus \{ 0\} $ has a simple pole at $z=0$, but its
*derivative
 \[
 f^*(z)=e^{f'(z)/f(z)}=e^{-1/z},\ z\in \mathbb{C}\setminus \{ 0\} ,
 \]
is holomorphic with an essential singularity at $z=0$. We also have $f^{*(n)}(z)=e^{(-1)^n(n-1)!/z^n}$, $z\in
\mathbb{C}\setminus \{ 0\} $, $n\in \mathbb{N}$.}
 \end{example}

The \emph{*derivative of a multi-valued function} can be defined as derivatives of its branches and it is naturally
expected to be multi-valued as in the following case.
 \begin{example}
 \label{E6}
{\rm The function $f(z)=\log z$, $z\in \mathbb{C}\setminus \{ 0\} $, is multi-valued with a branch point at $z=0$. Its
*derivative
 \[
 f^*(z)=e^{f'(z)/f(z)}=e^{1/(z\log z)},\ z\in \mathbb{C}\setminus \{ 0\},
 \]
is still multi-valued with a branch point at $z=0$}.
 \end{example}

In exceptional cases the *derivative of multi-valued function may be single-valued as in the following case.
 \begin{example}
 \label{E7}
{\rm The function $f(z)=e^{z\log z}$, $z\in \mathbb{C}\setminus \{ 0\} $, is multi-valued with a branch point at $z=0$,
but its
*derivative
 \[
 f^*(z)=e^{f'(z)/f(z)}=e^{1+\log z}=ez,\ z\in \mathbb{C}\setminus \{ 0\},
 \]
is single-valued and has a removable singularity at $z=0$, and its $n$th order *derivative
$f^{*(n)}(z)=e^{(-1)^n(n-2)!/z^{n-1}}$, $z\in \mathbb{C}\setminus \{ 0\}$, $n=2,3,\ldots \,$, accounts an essential
singularity at $z=0$. We also see that in complex *calculus the function $f(z)=e^{z\log z}$, $z\in \mathbb{C}\setminus
\{ 0\} $, plays the role of the quadratic function $g(z)=az^2$ with $a=e/2$ in Newtonian calculus.}
 \end{example}

We will say that a complex-valued function $f$ of complex variable (single- or multi-valued) is
\textit{*differentiable} at $z\in \mathbb{C}$ if it is differentiable at $z$ and $f(z)\not= 0$. We will also say that
$f$ is \textit{*holomorphic} or \textit{*analytic} on an open connected set $D$ if $f^*(z)$ exists for every $z\in D$.
The above examples demonstrate that for a given
*holomorphic function there are a few kinds of significant points on the complex plane which need a special attention:
 \begin{description}
 \item [{\rm (a)}] removable singular points, as in case of holomorphic functions,
 \item [{\rm (b)}] poles, as in case of holomorphic functions,
 \item [{\rm (c)}] essential singular points, as in case of holomorphic functions,
 \item [{\rm (d)}] zeros, at which a function takes zero value,
 \item [{\rm (e)}] branch point, if a function is multi-valued.
 \end{description}
Based on the above examples, we can state the following \emph{conservation principle} for the set $E_f$ of all these
points for the function $f$. Although *differentiation may change the kind of the points from $E_f$, the equality
$E_f=E_{f^{*(n)}}$ holds for all $n\in \mathbb{N}$. The points of $E_f$ are isolated, their total number is at most
countable and $E_f$ has no accumulation point according to the theorems of complex calculus. Note that there is no need
to make an exception for the identically zero function since it is not *holomorphic although it is entire. This basic
principle is not clearly seen in complex calculus and was not stated before to our knowledge. It become visible in
complex *calculus.

The following is an immediate consequence from the above discussion.
 \begin{theorem}
 \label{T2}
A complex-valued function $f$ of complex variable $z=x+iy$, defined on an open connected set $D$ of the complex plane,
is *holomorphic if and only if $f$ is nowhere-vanishing on $D$, $|f|$ and for every $z\in \mathbb{C}$ one of the local
continuous branches of $\arg f$ in a small neighborhood of $z$ have continuous partial derivatives in $x$ and in $y$,
and the Cauchy--Riemann *conditions $(\ref{7})$ hold. Furthermore, a *holomorphic function $f$ has all order
*derivatives and its $n$th order *derivative satisfies $(\ref{4})$.
 \end{theorem}


 \section{Properties of Complex Multiplicative Derivatives}

Some of properties of *derivatives from the real case, listed in Bashirov et al. \cite{B}, can immediately be extended
to complex case. For example,
 \begin{description}
 \item [{\rm (a)}]
 $[cf]^*(z)=f^*(z)$,\ $c=\mathrm{const.}\in \mathbb{C}\setminus \{ 0\} $;\vspace{0.1cm}
 \item [{\rm (b)}]
 $[fg]^*(z)=f^*(z)g^*(z)$;\vspace{0.1cm}
 \item [{\rm (c)}]
 $[f/g]^*(z)=f^*(z)/g^*(z)$,
 \end{description}
which can be proved directly by using (\ref{3}). But some other properties can be extended nontrivially. For example,
the equalities
 \[
 [f^g]^*(x)=f^*(x)^{g(x)}f(x)^{g'(x)},\ \ \text{and}\ \ [f\circ g]^*(x)=f^*(g(x))^{g'(x)}
 \]
from the real case fail in the complex case because they include multi-valued functions. They can be stated in the
form:
 \begin{description}
 \item [{\rm (d)}]
 $[e^{g\log f}]^*(z)\subseteq e^{g(z)\log f^*(z)}e^{g'(z)\log f(z)}$ in the sense each branch value of $[e^{g\log f}]^*(z)$ is the product of
 some branch values of $e^{g(z)\log f^*(z)}$ and $e^{g'(z)\log f(z)}$;\vspace{0.1cm}
 \item [{\rm (e)}]
 $[f\circ g]^*(z)\in e^{g'(z)\log f^*(g(z))}$ in the sense that $[f\circ g]^*(z)$ equals to some branch value of $e^{g'(z)\log f^*(g(z))}$.
 \end{description}
To prove (d), we evaluate
 \[
 \big[ e^{g\log f}\big] ^*(z)=e^{g'(z)\log f(z)}e^{g(z)f'(z)/f(z)}.
 \]
On the other hand,
 \[
 e^{g(z)\log f^*(z)}e^{g'(z)\log f(z)}=e^{g(z)\log e^{f'(z)/f(z)}}e^{g'(z)\log f(z)}.
 \]
Here, note that although the equality $e^{\log w}=w$ holds for all $w\not= 0$, in general we have $\log e^w\not =w$ if
we assume any branch of $\log $. The equality $\mathcal{L}e^w=w$ holds if $\mathcal{L}$ is a branch of $\log $, ranging
in the strip
 \[
 \{ x+iy\in \mathbb{C}:\alpha <y\le \alpha +2\pi \} ,
 \]
where the imaginary part of $w$ falls into $(\alpha ,\alpha +2\pi ]$. Letting $w=f'(z)/f(z)$, we obtain that each
branch value of $[e^{g\log f}]^*(z)$ is the product of some branch values of $e^{g(z)\log f^*(z)}$ and $e^{g'(z)\log
f(z)}$. This proves (d). In the same way, (e) can be proved.

In particular, if $g(z)\equiv c=\mathrm{const.}$ for $c\in\mathbb{C}$, then
 \[
 \big[ e^{c\log f}\big] ^*(z)=e^{cf'(z)/f(z)},
 \]
implying that $[e^{c\log f}]^*$ is single-valued. Hence, (d) reduces to
 \begin{description}
 \item [{\rm (f)}]
 $[e^{c\log f}]^*(z)\in e^{c\log f^*(z)}$, $c=\mathrm{const.}\in \mathbb{C}$, in the sense that $[e^{c\log f}]^*(z)$ equals to some branch value of
 $e^{c\log f^*(z)}$.
 \end{description}

The properties (d)--(f) are pointwise in sense that for distinct values of $z$ distinct branches of $\log $ are
required to get the respective branch values of the right hand sides. This makes (d)--(f) less useful than (a)--(c).
But for nonnegative integer values of $c$, both the right and left hand sides in (f) are single valued, reducing (f) to
the equality
 \begin{description}
 \item {\rm (g)}
 $[e^{n\log f}]^*(z)=e^{n\log f^*(z)}$, $n=\mathbb{N}\cup \{ 0\} $.
 \end{description}
The proof of this is just multiple application of (b).


 \section{Line Multiplicative Integrals}

In order to develop complex *integration, we need in line *integrals as well the fundamental theorem of calculus for
line *integrals.

Let $f$ be a positive function of two variables, defined on an open connected set in $\mathbb{R}^2$, and let $C$ be a
piecewise smooth curve in the domain of $f$. Take a partition $\mathcal{P}=\{ P_0,\ldots ,P_m\} $ on $C$ and let $(\xi
_k, \eta _k)$ be a point on $C$ between $P_{k-1}$ and $P_k$. Denote by $\Delta s_k$ the arclength of $C$ from the point
$P_{k-1}$ to $P_k$. According to the definition of *integral from Bashirov et al. \cite{B}, define the integral product
 \[
 P(f,\mathcal{P})=\prod _{k=1}^mf(\xi _k,\eta _k)^{\Delta s_k}.
 \]
The limit of this product when $\max \{ \Delta s_1,\ldots ,\Delta s_m\} \to 0$ independently on selection of the points
$(\xi _k,\eta _k)$ will be called a \emph{line *integral of $f$ in $ds$ along $C$}, for which we will use the symbol
 \[
 \int _Cf(x,y)^{ds}.
 \]
From
 \[
 \prod _{k=1}^mf(\xi _k,\eta _k)^{\Delta s_k}=e^{\sum _{k=1}^m\ln f(\xi _k,\eta _k)\Delta s_k},
 \]
it is clearly seen that the line *integral of $f$ along $C$ exists if $f$ is a positive function and the line integral
of $\ln f$ along $C$ exists, and they are related as
 \[
 \int _Cf(x,y)^{ds}=e^{\int _C\ln f(x,y)\,ds}.
 \]
The following properties of line *integrals in $ds$ can be proved easily:
 \begin{align*}
 {\rm (a)} &\ \int _C(f(x,y)^p)^{ds}=\Big( \int _Cf(x,y)^{ds}\Big) ^p,\ p\in \mathbb{R},
             \ \ \ \ \ \ \ \ \ \ \ \ \ \ \ \ \ \ \ \ \ \ \ \ \ \ \ \ \ \ \ \ \ \ \ \ \ \ \ \\
 {\rm (b)} &\ \int _C(f(x,y)g(x,y))^{ds}=\int _Cf(x,y)^{ds}\cdot \int _Cg(x,y)^{ds},\\
 {\rm (c)} &\ \int _C(f(x,y)/g(x,y))^{ds}=\int _Cf(x,y)^{ds}\big/ \int _Cg(x,y)^{ds},\\
 {\rm (d)} &\ \int _Cf(x,y)^{ds}=\int _{C_1}f(x,y)^{ds}\cdot \int _{C_2}f(x,y)^{ds},\ C=C_1+C_2,
 \end{align*}
where $C=C_1+C_2$ means that the curve $C$ is divided into two pieces at some interior point of $[a,b]$.

In a similar way, we can introduce the line *integrals in $dx$ and in $dy$ and establish their relation to the
respective line integrals in the form
 \begin{equation}
 \label{8}
 \int _Cf(x,y)^{dx}=e^{\int _C\ln f(x,y)\,dx}\ \ \text{and}\ \  \int _Cf(x,y)^{dy}=e^{\int _C\ln f(x,y)\,dy}.
 \end{equation}
Clearly, all these three forms of line *integrals exist if $f$ is a positive continuous function. The above mentioned
properties of the line *integrals in $ds$ are valid for line *integrals in $dx$ and in $dy$ as well. Additionally,
 \[
 \int _Cf(x,y)^{dx}=\Big( \int _{-C}f(x,y)^{dx}\Big) ^{-1}\ \ \text{and}\ \
 \int _Cf(x,y)^{dy}=\Big( \int _{-C}f(x,y)^{dy}\Big) ^{-1}
 \]
while
 \[
 \int _Cf(x,y)^{ds}=\int _{-C}f(x,y)^{ds},
 \]
where $-C$ is the curve $C$ with the opposite orientation. Moreover, the following evaluation formulae for the line
*integrals are also easily seen:
 \begin{align*}
 {\rm (a)} &\ \int _Cf(x,y)^{ds}&=\int _a^b\Big( f(x(t),y(t))^{\sqrt{x'(t)^2+y'(t)^2}}\Big) ^{dt},
              \ \ \ \ \ \ \ \ \ \ \ \ \ \ \ \ \ \ \ \ \ \ \ \ \ \ \ \ \ \ \ \\
 {\rm (b)} &\ \int _Cf(x,y)^{dx}&=\int _a^b\Big( f(x(t),y(t))^{x'(t)}\Big) ^{dt},
              \ \ \ \ \ \ \ \ \ \ \ \ \ \ \ \ \ \ \ \ \ \ \ \ \ \ \ \ \ \ \ \ \ \ \ \ \ \ \ \ \ \\
 {\rm (c)} &\ \int _Cf(x,y)^{dy}&=\int _a^b\Big( f(x(t),y(t))^{y'(t)}\Big) ^{dt},
              \ \ \ \ \ \ \ \ \ \ \ \ \ \ \ \ \ \ \ \ \ \ \ \ \ \ \ \ \ \ \ \ \ \ \ \ \ \ \ \ \ \\
 \end{align*}
where $\{ (x(t),y(t)):a\le t\le b\} $ is a suitable parametrization of $C$ and $\int _a^bg(t)^{dt}$ is the *integral of
$g$ on the interval $[a,b]$ (see, Bashirov et al. \cite{B}). It is also suitable to denote
 \[
 \int _Cf(x,y)^{dx}g(x,y)^{dy}=\int _Cf(x,y)^{dx}\cdot \int _Cg(x,y)^{dy}.
 \]
In cases when $C$ is a closed curve we write $\oint _C$ instead of $\int _C$.
 \begin{example}
 \label{E8}
{\rm Let $c>0$ and let $C=\{ (x(t),y(t)):a\le t\le b\} $ be a piecewise smooth curve. Then
 \[
 \int _Cc^{dx}=e^{\int _C\ln c\,dx}=e^{(x(b)-x(a))\ln c}=c^{x(b)-x(a)}.
 \] }
 \end{example}

 \begin{theorem}[Fundamental theorem of calculus for line *integrals]
 \label{T3}
Let $D\subseteq \mathbb{R}^2$ be an open connected set and let $C=\{ (x(t),y(t)):a\le t\le b\} $ be a piecewise smooth
curve in $D$. Assume that $f$ is a real-valued continuously *differentiable function on $D$. Then
 \[
 \int _Cf^*_x(x,y)^{dx}f^*_y(x,y)^{dy}=f(x(b),y(b))/f(x(a),y(a)).
 \]
 \end{theorem}
 \begin{proof}
From the fundamental theorem of calculus for line integrals,
 \begin{align*}
 \int _Cf^*_x(x,y)^{dx}f^*_y(x,y)^{dy}
  & =e^{\int _C(\ln f^*_x(x,y)\,dx+\ln f^*_y(x,y)\,dy)}\\
  & =e^{\int _C([\ln f]'_x(x,y)\,dx+[\ln f]'_y(x,y)\,dy)}\\
  & =e^{\ln f(x(b),y(b))-\ln f(x(a),y(a))}\\
  & =f(x(b),y(b))/f(x(a),y(a)),
 \end{align*}
proving the theorem.
 \end{proof}

As far as line *integrals are concerned, we can present another fundamental theorem of *calculus related to line
*integrals, that is the Green's theorem in *form. Let $f$ be a bounded positive function $f$, defined
on the Jordan set $D\subseteq \mathbb{R}^2$. Let $\mathcal{Q}=\{D_k:k=1,\ldots ,m\} $ be a partition of $D$. Take any
$(\xi _k,\eta _k)\in D_k$ and let $A_k$ be the area of $D_k$. Define the integral product
 \[
 P(f,\mathcal{Q})=\prod _{k=1}^mf(x(t),y(t))^{A_k}.
 \]
The limit of this product when $\max \{ A_1,\ldots ,A_m\} \to 0$ independently on selection of the points $(\xi _k,\eta
_k)$ will be called a \emph{double *integral of $f$ on $D$}, for which we will use the symbol
 \[
 \int \!\!\int _Df(x,y)^{dA}.
 \]
A relation between double integrals and double *integrals can be easily derived as
 \[
 \int \!\!\int _Df(x,y)^{dA}=e^{\int \!\!\int _Df(x,y)dA}.
 \]
The following properties of double *integrals can also be proved easily:
 \begin{align*}
 {\rm (a)} &\ \int \!\!\int _D(f(x,y)^p)^{dA}=\Big( \int \!\!\int _Df(x,y)^{dA}\Big) ^p,\ p\in \mathbb{R},
             \ \ \ \ \ \ \ \ \ \ \ \ \ \ \ \ \ \ \ \ \ \ \ \ \ \ \ \ \ \ \ \ \ \ \ \ \ \ \ \\
 {\rm (b)} &\ \int \!\!\int _D(f(x,y)g(x,y))^{dA}=\int \!\!\int _Df(x,y)^{dA}\cdot \int \!\!\int _Dg(x,y)^{dA},\\
 {\rm (c)} &\ \int \!\!\int _D(f(x,y)/g(x,y))^{dA}=\int \!\!\int _Df(x,y)^{dA}\big/ \int \!\!\int _Dg(x,y)^{dA},\\
 {\rm (d)} &\ \int \!\!\int _Df(x,y)^{dA}=\int \!\!\int _{D_1}f(x,y)^{dA}\cdot \int \!\!\int _{D_2}f(x,y)^{dA},\ D=D_1+D_2,
 \end{align*}
where $D=D_1+D_2$ means that $D_1$ and $D_2$ are two non-overlapping Jordan sets with $D_1\cup D_2=D$.
 \begin{theorem}[Green's theorem in *form]
 \label{T4}
Let $f$ and $g$ be real-valued continuously *differentiable functions on a simply connected Jordan set $D\subseteq
\mathbb{R}^2$ with the piecewise smooth and positively oriented boundary $C$. Then
 \[
 \oint _Cf(x,y)^{dx}g(x,y)^{dy}=\int \!\! \int _D(g^*_x(x,y)/f^*_y(x,y))^{dA}.
 \]
 \end{theorem}
 \begin{proof}
From the Green's theorem,
 \begin{align*}
 \oint _Cf(x,y)^{dx}g(x,y)^{dy}
  & =e^{\oint _C(\ln f(x,y)\,dx+\ln g(x,y)\,dy)}\\
  & =e^{\int \!\!\int _D([\ln g]'_x(x,y)-[\ln f]'_y(x,y))\,dA}\\
  & =e^{\int \!\!\int _D(\ln g^*_x(x,y)-\ln f^*_y(x,y))\,dA}\\
  & =\int \!\!\int _D(g^*_x(x,y)/f^*_y(x,y))^{dA},
 \end{align*}
proving the theorem.
 \end{proof}


 \section{Complex Multiplicative Integration}

Let $f$ be a continuous nowhere-vanishing complex-valued function of complex variable and let $z(t)=x(t)+iy(t)$, $a\le
t\le b$, be a complex-valued function of real variable, tracing a piecewise smooth simple curve $C$ in the open
connected domain $D$ of $f$. The complex *integral of $f$ along $C$ will heavily use $\log f$. In order to represent
$\log f$ as the composition of branches of $\log $ and $f$ along the whole curve $C$ we will use a ``method of
localization" from Sarason \cite{S}. In this section we will consider a simple case assuming that the length of the
interval $[a,b]$ is sufficiently small so that all the values of $f(z(t))$ for $a\le t\le b$ fall into an open half
plane bounded by a line through the origin. Under this condition the restriction of $\log f$ to $C$ can be treated as a
composition of the branches of $\log $ and the restriction of $f$ to $C$. Moreover, we can select one of the
multi-values of $\log f(z(a))$ and consider a branch $\mathcal{L}$ of $\log $ so that $\mathcal{L}(f(z(a)))$ equals to
this preassigned value. Thus
 \[
 \log f(z(t))=\mathcal{L}(f(z(t)))+2\pi ni,\ a\le t\le b,\ n\in \mathbb{Z}.
 \]

Now, take a partition $\mathcal{P}=\{ z_0,\ldots ,z_m\} $ on $C$ and let $\zeta _k$ be a point on $C$ between $z_{k-1}$
and $z_k$. Denote $\Delta z_k=z_k-z_{k-1}$. Consider the integral product $\prod _{k=1}^me^{\Delta z_k\log f(\zeta
_k)}$. It can be evaluated as
 \begin{align*}
 \prod _{k=1}^me^{\Delta z_k\log f(\zeta _k)}
  &=e^{\sum _{k=1}^m(\mathcal{L}(f(\zeta _k))+2\pi ni)\Delta z_k}\\
  &=e^{2\pi n(z(b)-z(a))i}e^{\sum _{k=1}^m\mathcal{L}(f(\zeta _k))\Delta z_k},\ n\in \mathbb{Z},
 \end{align*}
showing that $\prod _{k=1}^me^{\Delta z_k\log f(\zeta _k)}$ has more than one value. Let
 \begin{equation}
 \label{9}
 P_0(f,\mathcal{P})=e^{\sum _{k=1}^m\mathcal{L}(f(\zeta _k))\Delta z_k}
 \end{equation}
and let
 \begin{equation}
 \label{10}
 P_n(f,\mathcal{P})=e^{2\pi n(z(b)-z(a))i}P_0(f,\mathcal{P}),\ n\in \mathbb{Z}.
 \end{equation}
The limit of $P_0(f,\mathcal{P})$ as $\max \{ |\Delta z_1|,\ldots ,|\Delta z_m|\} \to 0$ independently on selection of
the points $\zeta _k$ will be called a \emph{branch value of the complex *integral of $f$ along $C$} and it will be
denoted by $I^*_0(f,C)$. Then \emph{the complex *integral of $f$ along $C$} can be defined as the multiple values
 \begin{equation}
 \label{11}
 I^*_n(f,C)=e^{2\pi n(z(b)-z(a))i}I^*_0(f,C),\ n\in \mathbb{Z},
 \end{equation}
which will be denoted by
 \[
 \int _Cf(z)^{dz}.
 \]
Note that if $z(b)-z(a)$ is an integer, then all the values of $\int _Cf(z)^{dz}$ equal to $I^*_0(f,C)$, i.e.,
$I^*(f,C)$ become single-valued. If $z(b)-z(a)$ is a rational number in the form $p/q$, where $p$ and $q$ are
irreducible integers, then $\int _Cf(z)^{dz}$ has $q$ distinct values
 \[
 e^{2\pi npi/q}I^*_0(f,C),\ n=0,1,\ldots ,q-1 .
 \]
Generally, $\int _Cf(z)^{dz}$ has countably many distinct values. In case if $z(b)-z(a)$ is a real number, we also have
$|I^*_n(f,C)|=|I^*_0(f,C)|$ for all $n$. Similarly, if $z(b)-z(a)$ is an imaginary number, $\mathrm{Arg}\,
I^*_n(f,C)=\mathrm{Arg}\, I^*_0(f,C)$ for all $n$.

The existence of the complex *integral of $f$ can be reduced to the existence of line *integrals in the following way.
Let $R(z)=|f(z)|$ and $\Theta (z)=\mathrm{Im}\,\mathcal{L}(f(z))$ for $z\in C$. Denote $z=x+iy$ and $\Delta z_k=\Delta
x_k+i\Delta y_k$. Then from (\ref{9}),
 \begin{align*}
 P_0(f,\mathcal{P})
  &=e^{\sum _{k=1}^m\mathcal{L}(f(\zeta _k))\Delta z_k}\\
  &=e^{\sum _{k=1}^m(\ln R(\zeta _k)+i\Theta (\zeta _k))(\Delta x_k+i\Delta y_k)}\\
  &=e^{\sum _{k=1}^m(\ln R(\zeta _k)\Delta x_k-\Theta (\zeta _k)\Delta y_k)+i\sum _{k=1}^m(\Theta (\zeta _k)\Delta x_k+\ln R(\zeta _k)\Delta y_k)}.
 \end{align*}
If the limits of the sums in the last expression exist, then they are line integrals, producing
 \begin{equation}
 \label{12}
 I^*_0(f,C)=e^{\int _C(\ln R(z)\,dx-\Theta (z)\,dy)+i\int _C(\Theta(z)\,dx+\ln R(z)\,dy)}.
 \end{equation}
Additionally, by Example \ref{E8},
 \[
 e^{2\pi n(z(b)-z(a))i}=e^{2\pi n(-(y(b)-y(a))+i(x(b)-x(a)))}=e^{-\int _C2\pi n\,dy+i\int _C2\pi n\,dx}.
 \]
By (\ref{11})--(\ref{12}), this implies
 \begin{equation}
 \label{13}
 I^*_n(f,C)=e^{\int _C(\ln R(z)\,dx-(\Theta (z)+2\pi n)\,dy)+i\int _C((\Theta(z)+2\pi n)\,dx+\ln R(z)\,dy)}
 \end{equation}
for $n\in \mathbb{Z}$ or, in the multi-valued form,
 \[
 \int _Cf(z)^{dz}=e^{\int _C\log f(z)\, dz},
 \]
in which
 \begin{equation}
 \label{14}
 I^*_0(f,C)=e^{\int _C\mathcal{L}(f(z))\,dz}\ \ \text{and}\ \ I^*_n(f,C)=e^{2\pi n(z(b)-z(a))i}I^*_0(f,C).
 \end{equation}

To write (\ref{13}) in terms of line *integrals, note that by (\ref{8})
 \[
 |I^*_n(f,C)| =e^{\int _C(\ln R(z)\,dx-(\Theta (z)+2\pi n)\,dy)}=\int _CR(z)^{dx}\big( e^{-\Theta (z)-2\pi n}\big) ^{dy}
 \]
and
 \begin{align*}
 \arg I^*_n(f,C)
  & =\int _C((\Theta(z)+2\pi n)\,dx+\ln R(z)\,dy)+2\pi m\\
  & =\ln \int _C\big( e^{\Theta (z)+2\pi n}\big) ^{dx}R(z)^{dy}+2\pi m.
 \end{align*}
Hence,
 \begin{equation}
 \label{15}
 I^*_n(f,C)=\int _CR(z)^{dx}\big( e^{-\Theta (z)-2\pi n}\big) ^{dy}e^{i\ln \int _C\left( e^{\Theta (z)+2\pi n}\right) ^{dx}R(z)^{dy}}
 \end{equation}
for $n\in \mathbb{Z}$. Thus, the conditions, imposed at the beginning of this section, namely, (a) $f$ is
nowhere-vanishing and continuous on the open connected set $D$, (b) $C$ is a piecewise simple smooth curve in $D$, and
(c) $\{ f(z(t)):a\le t \le b\} $ falls into an open half plane bounded by a line through the origin, guarantee the
existence of $\int _Cf(z)^{dz}$ as multiple value.

The following proposition will be used for justifying the correctness of the definition of the complex *integral for
arbitrary interval $[a,b]$ in the next section.
 \begin{proposition}[1st multiplicative property, local]
 \label{P2}
Let $f$ be a nowhere-vanishing continuous function, defined on an open connected set $D$, and let $C=\{
z(t)=x(t)+iy(t):a\le t\le b\} $ be a piecewise smooth simple curve in $D$ with the property that the set $\{
f(z(t)):a\le t\le b\} $ falls into an open half plane bounded by a line through origin. Take any $a<c<b$ and let
$C_1=\{ z(t)=x(t)+iy(t):a\le t\le c\} $ and $C_2=\{ z(t)=x(t)+iy(t):c\le t\le b\} $. Then
 \[
 \int _Cf(z)^{dz}=\int _{C_1}f(z)^{dz}\int _{C_2}f(z)^{dz},
 \]
where the equality is understood in the sense that
 \[
 I^*_n(f,C)=I^*_n(f,C_1)I^*_n(f,C_2)\ \ \text{for all}\ \ n\in \mathbb{Z}
 \]
with the same branch $\mathcal{L}$ of $\log $ used for $I_0(f,C)$, $I_0(f,C_1)$ and $I_0(f,C_2)$.
 \end{proposition}
 \begin{proof}
This follows immediately from (\ref{13}) and the respective property of line integrals.
 \end{proof}

Next, we consider a *analog of the fundamental theorem of complex calculus in a local form.
 \begin{proposition}[Fundamental theorem of complex *calculus, local]
 \label{P3}
Let $f$ be a nowhere-vanishing *holomorphic function, defined on an open connected set $D$, and let $C=\{
z(t)=x(t)+iy(t):a\le t\le b\} $ be a piecewise smooth simple curve in $D$ with the property that the set $\{
f(z(t)):a\le t\le b\} $ falls into an open half plane bounded by a line through origin. Then
 \[
 \int _Cf^*(z)^{dz}=\big\{ e^{2\pi n(z(b)-z(a))i}f(z(b))/f(z(a)): n\in \mathbb{Z} \big\} .
 \]
 \end{proposition}
 \begin{proof}
From (\ref{15}),
 \[
 \int _Cf^*(z)^{dz}=\int _C|f^*(z)|^{dx}\big( e^{-\arg f^*(z)}\big) ^{dy}\ e^{i\ln \int _C\left( e^{\arg f^*(z)}\right) ^{dx}|f^*(z)|^{dy}}.
 \]
Using (\ref{7}),
 \begin{align*}
 I_n^*(f^*,C)
  &=\int _CR^*_x(z)^{dx}\big( e^{[\ln R]'_y(z)-2\pi n}\big) ^{dy}\ e^{i\ln \int _C\left( e^{\Theta '_x(z)+2\pi n}\right) ^{dx}
    \left[ e^\Theta \right] ^*_y(z)^{dy}}\\
  &=\int _CR^*_x(z)^{dx}R^*_y(z)^{dy}\ e^{i\ln \int _C\left[ e^\Theta \right] ^*_x(z)^{dx}\left[ e^{\Theta }\right] ^*_y(z)^{dy}}\\
  &\ \ \ \times  \int _C\big( e^{-2\pi n}\big) ^{dy}\ e^{i\ln \int _C(e^{2\pi n})^{dx}}.
 \end{align*}
By Theorem \ref{T3},
 \[
 \int _CR^*_x(z)^{dx}R^*_y(z)^{dy}\ e^{i\ln \int _C\left[ e^\Theta \right] ^*_x(z)^{dx}\left[ e^{\Theta }\right] ^*_y(z)^{dy}}
 =\frac{R(z(b))e^{i\Theta (z(b))}}{R(z(a))e^{i\Theta (z(a))}}=\frac{f(z(b))}{f(z(a))},
 \]
and, by Example \ref{E8},
 \[
 \int _C\big( e^{-2\pi n}\big) ^{dy}\ e^{i\ln \int _C(e^{2\pi n})^{dx}}=e^{2\pi n(-(y(b)+y(a))+i(x(b)-x(a)))}=e^{2\pi n(z(b)-z(a))i}.
 \]
Thus, the proposition is proved.
 \end{proof}


 \section{Complex Multiplicative Integration (Continued)}

The definition of complex *integral from the previous section does not cover the important closed curves about the
origin, say, unit circle centered at the origin, because such curves do not fall into any open half plane bounded by a
line through origin. In this section this restriction will be removed.
 \begin{lemma}
 \label{L1}
Let $f$ be a continuous nowhere-vanishing function, defined on an open connected set $D$, and let $C=\{
z(t)=x(t)+iy(t):a\le t\le b\} $ be a piecewise smooth simple curve in $D$. Then there exists a partition
$\mathcal{P}=\{ t_0,t_1,\ldots ,t_m\} $ of $[a,b]$ such that each of the sets $\{ f(z(t)):t_{k-1}\le t\le t_k\} $,
$k=1,\ldots ,m$, falls into an open half plane bounded by a line through origin.
 \end{lemma}
 \begin{proof}
To every $t\in [a,b]$, consider $\theta _t=\mathrm{Arg}\,f(z(t))$ and the line $L_t$ formed by the rays $\theta =\theta
_t+\pi/2$ and $\theta =\theta _t-\pi/2$. Since $f$ is continuous and nowhere-vanishing, there is an interval
$(t-\varepsilon ,t+\varepsilon )\subseteq [a,b]$ such that the set
 \[
 \{ f(z(s)):t-\varepsilon <s<t+\varepsilon \}
 \]
falls into one of the open half planes bounded by $L_t$ if $t\in (a,b)$. In case of $t=a$ such an interval can be
selected in the form $[a,a+\varepsilon )$ and in the case $t=b$ as $(b-\varepsilon ,b]$. The collection of all such
intervals forms an open cover of the compact subspace $[a,b]$ of $\mathbb{R}$. Therefore, there is a finite number of
them covering $[a,b]$. Writing the end points of these intervals in an increasing order $a=t_0<t_1<\cdots <t_m=b$
produces a required partition.
 \end{proof}

This lemma determines a way for the definition of $\int _Cf(z)^{dz}$ in the general case. Assume again that $f$ is a
continuous nowhere-vanishing complex-valued function of complex variable and $z(t)=x(t)+iy(t)$, $a\le t\le b$, is a
complex-valued function of real variable, tracing a piecewise smooth simple curve $C$ in the open connected domain $D$
of $f$. Let $\mathcal{P}=\{ t_0,t_1,\ldots ,t_m\} $ be a partition of $[a,b]$ from Lemma \ref{L1} and let $C_k=\{
z(t):t_{k-1}\le t\le t_k\} $. Choose any branch $\mathcal{L}_1$ of $\log $ and consider $\int _{C_1}f(z)^{dz}$ as
defined in the previous section. Then select a branch $\mathcal{L}_2$ of $\log $ with
$\mathcal{L}_2(f(z(t_1)))=\mathcal{L}_1(f(z(t_1)))$ and consider $\int _{C_2}f(z)^{dz}$. Next, select a branch
$\mathcal{L}_3$ of $\log $ with $\mathcal{L}_3(f(z(t_2)))=\mathcal{L}_2(f(z(t_2)))$ and consider $\int
_{C_3}f(z)^{dz}$, etc. In this process the selection of the starting branch $\mathcal{L}_1$ is free, but the other
branches $\mathcal{L}_2,\ldots ,\mathcal{L}_m$ are selected accordingly to construct a continuous single-valued
function $g$ on $[a,b]$ such that the value of $g$ at fixed $t\in [a,b]$ equals to one of the branch values of $\log
f(z(t))$. Following (\ref{14}), the \emph{complex *integral of $f$ over $C$}, that will again be denoted by $\int
_Cf(z)^{dz}$, can be defined as the multiple values
 \[
 I^*_n(f,C)=\prod _{k=1}^me^{2\pi n(z(t_k)-z(t_{k-1}))i+\int _{C_k}\mathcal{L}_k(f(z))\,dz},\ n\in \mathbb{Z},
 \]
or
 \begin{equation}
 \label{16}
 I^*_n(f,C)=e^{2\pi n(z(b)-z(a))i}e^{\sum _{k=1}^m\int _{C_k}\mathcal{L}_k(f(z))\,dz},\ n\in \mathbb{Z}.
 \end{equation}

This definition is independent on the selection of the partition $\mathcal{P}$ of $[a,b]$. Indeed, if $\mathcal{Q}$ is
another partition, being a refinement of the previous one, then the piece $C_k$ from $z(t_{k-1})$ to $z(t_k)$ of the
curve $C$ became departed into smaller non-overlapping pieces $C_{ki}$, $i=1,\ldots ,l_k$, each over the partition
intervals of $\mathcal{Q}$ falling into $[t_{k-1},t_k]$. Since the range of $f$ over $C_k$ falls into an open half
plane bounded by a line through origin, the range of $f$ over each $C_{ki}$ falls into the same half plane. Therefore,
by Proposition \ref{P2},
 \[
 I^*_n(f,C_k)=\prod _{i=1}^{l_k}I^*_n(f,C_{ki})
 \]
with the same branch $\mathcal{L}_k$ of $\log $ used for all $I_0(f,C_k)$ and $I_0(f,C_{k1}),\ldots ,I_0(f,C_{kl_k})$.
Then
 \[
 I^*_n(f,C_k)=\prod _{k=1}^{m}I^*_n(f,C_k))=\prod _{k=1}^{m}\prod _{i=1}^{l_k}I^*_n(f,C_{ki}),
 \]
i.e., both $\mathcal{P}$ and $\mathcal{Q}$ produce the same multiple values. In case if $\mathcal{P}$ and $\mathcal{Q}$
are two arbitrary partitions of $[a,b]$ from Lemma \ref{L1}, we can compare the integral for selections $\mathcal{P}$
and $\mathcal{Q}$ with the same for their refinement $\mathcal{P}\cup \mathcal{Q}$ and deduce that $I^*_n(f,C)$ is
independent on selection of $\mathcal{P}$ and $\mathcal{Q}$.


 \section{Properties of Complex Multiplicative Integrals}

 \begin{theorem}[1st multiplicative property]
 \label{T5}
Let $f$ be a nowhere-vanishing continuous function, defined on an open connected set $D$, and let $C=\{
z(t)=x(t)+iy(t):a\le t\le b\} $ be a piecewise smooth simple curve in $D$. Take any $a<c<b$ and let $C_1=\{
z(t)=x(t)+iy(t):a\le t\le c\} $ and $C_2=\{ z(t)=x(t)+iy(t):c\le t\le b\} $. Then
 \[
 \int _Cf(z)^{dz}=\int _{C_1}f(z)^{dz}\int _{C_2}f(z)^{dz}
 \]
in the sense that $I^*_n(f,C)=I^*_n(f,C_1)I^*_n(f,C_2)$, $n\in \mathbb{Z}$, where
$\mathcal{L}_{01}(f(z(a)))=\mathcal{L}_{11}(f(z(a)))$ and $\mathcal{L}_{1m_1}(f(z(c)))=\mathcal{L}_{21}(f(z(c)))$ if
$\mathcal{L}_{01},\ldots ,\mathcal{L}_{0m}$, $\mathcal{L}_{11},\ldots ,\mathcal{L}_{1m_1}$ and $\mathcal{L}_{21},\ldots
,\mathcal{L}_{2m_2}$ are the sequences of branches of $\log $ used in definition of $I^*_0(f,C)$, $I^*_0(f,C_1)$ and
$I^*_0(f,C_2)$, respectively.
 \end{theorem}
 \begin{proof}
This follows from the definition of complex *integral for general interval $[a,b]$ and its independence on the
selection of partition $\mathcal{P}$ from Lemma \ref{L1}.
 \end{proof}
 \begin{theorem}[2nd multiplicative property]
 \label{T6}
Let $f$ and $g$ be nowhere-vanishing continuous functions, defined on an open connected set $D$, and let $C=\{
z(t)=x(t)+iy(t):a\le t\le b\} $ be a piecewise smooth simple curve in $D$. Then
 \[
 \int _C(f(z)g(z))^{dz}=\int _Cf(z)^{dz}\int _Cg(z)^{dz}
 \]
as a set equality, where the product of the sets $A$ and $B$ is treated as $AB=\{ ab: a\in A,\ b\in B\} $.
 \end{theorem}
 \begin{proof}
This follows from (\ref{16}) and the set equality $\log (z_1z_2)=\log z_1+\log z_2$, where the sum of the sets $A$ and
$B$ is treated as $A+B=\{ a+b: a\in A,\ b\in B\} $.
 \end{proof}
 \begin{theorem}[Division property]
 \label{T7}
Let $f$ and $g$ be nowhere-vanishing continuous functions, defined on an open connected set $D$, and let $C=\{
z(t)=x(t)+iy(t):a\le t\le b\} $ be a piecewise smooth simple curve in $D$. Then
 \[
 \int _C(f(z)/g(z))^{dz}=\int _Cf(z)^{dz}\big/ \int _Cg(z)^{dz}
 \]
as a set equality, where the ratio of the sets $A$ and $B$ is treated as $A/B=\{ a/b: a\in A,\ b\in B\} $.
 \end{theorem}
 \begin{proof}
This follows from (\ref{16}) and the set equality  $\log (z_1/z_2)=\log z_1-\log z_2$, where the difference of the sets
$A$ and $B$ is treated as $A-B=\{ a-b: a\in A,\ b\in B\} $.
 \end{proof}
 \begin{theorem}[Reversing the curve]
 \label{T8}
Let $f$ be a nowhere-vanishing continuous functions, defined on an open connected set $D$, let $C=\{
z(t)=x(t)+iy(t):a\le t\le b\} $ be a piecewise smooth simple curve in $D$ and let $-C$ be the curve $C$ with opposite
orientation. Then
 \[
 \int _Cf(z)^{dz}=\Big( \int _{-C}f(z)^{dz}\Big) ^{-1}
 \]
in the sense that $I^*_n(f,C)=I^*_n(f,-C)^{-1}$, $n\in \mathbb{Z}$, where
$\mathcal{L}_{11}(f(z(a)))=\mathcal{L}_{2m_2}(f(z(a)))$ or $\mathcal{L}_{1m_1}(f(z(b)))=\mathcal{L}_{21}(f(z(b)))$ if
the sequences of branches of $\log $ used in the definitions of $I^*_0(f,C)$ and $I^*_0(f,-C)$ are
$\mathcal{L}_{11},\ldots ,\mathcal{L}_{1m_1}$ and $\mathcal{L}_{21},\ldots ,\mathcal{L}_{2m_2}$, respectively.
 \end{theorem}
 \begin{proof}
This follows from (\ref{16}).
 \end{proof}
 \begin{theorem}[Raising to a natural power]
 \label{T9}
Let $f$ be a nowhere-vanishing continuous function, defined on an open connected set $D$ and let $C=\{
z(t)=x(t)+iy(t):a\le t\le b\} $ be a piecewise smooth simple curve in $D$. Then for $n\in \mathbb{Z}$,
 \[
 \Big( \int _Cf(z)^{dz}\Big) ^n\subseteq \int _C(f(z)^n)^{dz}.
 \]
 \end{theorem}
 \begin{proof}
This follows from multiple application of Theorem \ref{T6} and the fact that $A^n\subseteq AA\cdots A$ ($n$ times),
where the set $A^n$ is treated as $A^n=\{ a^n: a\in A\} $, but $AA\cdots A(n\ \text{times})=\{ a_1a_2\cdots a_n:a_i\in
A, i=1,\ldots ,n\} $.
 \end{proof}
 \begin{theorem}[Fundamental theorem of calculus for complex *integrals]
 \label{T10}
Let $f$ be a nowhere-vanishing *holomorphic function, defined on an open connected set $D$, and let $C=\{
z(t)=x(t)+iy(t):a\le t\le b\} $ be a piecewise smooth simple curve in $D$. Then
 \begin{equation}
 \label{17}
 \int _Cf^*(z)^{dz}=\big\{ e^{2\pi n(z(b)-z(a))i}f(z(b))/f(z(a)),\ n\in \mathbb{Z}\big\} .
 \end{equation}
 \end{theorem}
 \begin{proof}
Let $\mathcal{P}=\{ t_0,t_1,\ldots ,t_m\} $ be a partition from Lemma \ref{L1} and let $C_k=\{ z(t):t_{k-1}\le t\le
t_k\} $. Then
 \[
 \int _Cf(z)^{dz}=\int _{C_1}f(z)^{dz}\cdots \int _{C_m}f(z)^{dz}.
 \]
Hence, by Proposition \ref{P3},
 \begin{align*}
 I^*_n(f,C)
  & =e^{2\pi n(z(t_1)-z(t_0))i+\cdots +2\pi  n(z(t_m)-z(t_{m-1}))i}\frac{f(z(t_1))\cdots f(z(t_m))}{f(z(t_0))\cdots  f(z(t_{m-1}))}\\
  & =e^{2\pi n(z(b)-z(a))i}\frac{f(z(b))}{f(z(a))},
 \end{align*}
proving the theorem.
 \end{proof}

This theorem demonstrates that $\int _Cf^*(z)^{dz}$ is independent of the shape of the piecewise smooth curve $C$, but
depends on its initial point $z(a)=x(a)+iy(a)$ and end point $z(b)=x(b)+iy(b)$ on the curve $C$. Therefore, this
integral can be denoted by
 \[
 \int _{z(a)}^{z(b)}f^*(z)^{dz}.
 \]
 \begin{corollary}
 \label{C1}
Let $f$ be a nowhere-vanishing *holomorphic function, defined on an open connected set $D$, and let $C=\{
z(t)=x(t)+iy(t):a\le t\le b\} $ be a piecewise smooth simple closed curve in $D$. Then
 \begin{equation}
 \label{18}
 \oint _Cf^*(z)^{dz}=1.
 \end{equation}
 \end{corollary}
 \begin{proof}
Simply, write $z(a)=z(b)$ in (\ref{17}).
 \end{proof}

Note that in (\ref{18}) all the values of $\oint _Cf^*(z)^{dz}$ are equal to 1, i.e., $\oint _Cf^*(z)^{dz}$ becomes
single-valued.
 \begin{example}
 \label{E9}
{\rm By Example \ref{E2}, the function $f(z)=e^{cz}$, $z\in \mathbb{C}$, where $c=\mathrm{const}\in \mathbb{C}$, has the
*derivative $f^*(z)=e^c$. Respectively,
 \[
 \int _C(e^c)^{dz}=e^{2\pi n(z(b)-z(a))i}e^{c(z(b)-z(a))}=e^{(z(b)-z(a))(c+2\pi ni)},
 \]
where $C=\{ z(t):a\le t\le b\} $ is a piecewise smooth curve. }
 \end{example}
 \begin{example}
 \label{E10}
{\rm By Example \ref{E3}, the function $f(z)=e^{ce^z}$, $z\in \mathbb{C}$, where $c=\mathrm{const}\in \mathbb{C}$, has
the
*derivative $f^*(z)=f(z)$. Respectively,
 \[
 \int _C\big( e^{ce^z}\big) ^{dz}=e^{2\pi n(z(b)-z(a))i}e^{c\left( e^{z(b)}-e^{z(a)}\right) },
 \]
where again $C=\{ z(t):a\le t\le b\} $ is a piecewise smooth curve.
 }
 \end{example}
 \begin{example}
 \label{E11}
{\rm The analog of the integral
 \[
 \oint _{|z|=1}\frac{dz}{z}=2\pi i
 \]
in complex *calculus is
 \[
 \oint _{|z|=1}\big( e^{1/z}\big) ^{dz}.
 \]
Assuming that the orientation on the unit circle $|z|=1$ is positive, we informally have
 \[
 \oint _{|z|=1}\big( e^{1/z}\big) ^{dz}=e^{\oint _{|z|=1}\log e^{1/z}dz}=e^{\oint _{|z|=1}\big( \frac{1}{z}+2\pi ni\big)dz}
 =e^{2\pi n(z(b)-z(a))i}e^{2\pi i}=1.
 \]
Formally, we use Example \ref{E4} and calculate the same:
 \[
 \oint _{|z|=1}\big( e^{1/z}\big) ^{dz}=e^{2\pi n(z(b)-z(a))i}\frac{z(b)}{z(a)}=\frac{e^{\pi i}}{e^{-\pi i}}=e^{2\pi i}=1.
 \]
Thus, this example also fits to Corollary \ref{C1}. The main idea of this is that $e^0=e^{2\pi i}=1$ though $2\pi
i\not= 0$. In other words, the discontinuity of $\log $ on $(-\infty ,0]$, that creates the Cauchy formula $\int
_{|z|=1}\frac{dz}{z}=2\pi i$, appears in a smooth form in complex *calculus because now $\log z$ is replaced by
$e^{\log z}=z$, where the discontinuity of $\log $ is compensated by periodicity of the exponential function.
 }
 \end{example}


 \section{Conclusion}

In the paper the concepts of multiplicative derivative and multiplicative integral are extended to complex-valued
functions of complex variable. Some open questions from the real case are answered satisfactorily. Unlike the real
case, complex multiplicative calculus is found to be not so much parallel to ordinary complex calculus. In particular,
the complex multiplicative integral is defined as multiple values whereas the complex multiplicative derivative as a
single value. More importantly, the famous Cauchy formula $\oint \frac{dz}{z}=2\pi i$ of complex calculus produces no
effect in the multiplicative case. The reason is that in the multiplicative case the inequality $2\pi i\not= 0$ appears
as the equality $e^{2\pi i}=e^0=1$. This demonstrates the importance of further study in complex multiplicative
calculus. It is of a great importance the Cauchy's *theorem in general form and residue theory in multiplicative case.
We expect that it can be more natural tool for study of infinite products of complex functions.


 \section*{Acknowledgement}

The authors thank Prof. Sergey Khrushchev for useful discussions and suggestions.\\
This work is part of the B-type project MEKB-09-05.

  \end{document}